\documentclass[11pt, oneside, a4paper]{article}
\usepackage[latin1]{inputenc}
\usepackage{amsmath,amssymb,amsthm,mathrsfs,epsfig,color,graphicx}
\usepackage{hyperref}
\newtheorem{theorem}{Theorem}
\newtheorem{proposition}[theorem]{Proposition}
\newtheorem{lemma}[theorem]{Lemma}
\newtheorem{definition}[theorem]{Definition}
\newtheorem{corollary}[theorem]{Corollary}

\definecolor{Red}{cmyk}{0,1,1,0}

\definecolor{Blue}{cmyk}{1,1,0,0}

\newcommand{\ba}{\begin{array}}
\newcommand{\ea}{\end{array}}
\newcommand{\be}{\begin{equation}}
\newcommand{\ee}{\end{equation}}
\newcommand{\ben}{\begin{enumerate}}
\newcommand{\een}{\end{enumerate}}

\newcommand{\FX}{F\langle X\rangle}
\newcommand{\R}{\mathbb{R}}
\newcommand{\C}{\mathbb{C}}

\begin{document}

\title{Multihomogeneous Normed Algebras and \\ Polynomial Identities}

\author{
Leandro Cioletti\thanks{
 \texttt{cioletti@mat.unb.br}}
\\
Jos\'e Ant\^onio Freitas\thanks{ \texttt{jfreitas@mat.unb.br}; Partially supported by grant from CNPq No. 478318/2010-3
}
\\
\\
Departamento de Matem\'atica,\\
Universidade de Bras\'\i lia,\\
70910-900 Bras\'\i lia, DF, Brazil
\\
\\
Dimas Jos\'e Gon\c{c}alves\thanks{\texttt{dimas@dm.ufscar.br}; Partially supported by grant from CNPq No. 478318/2010-3}
\\
\\
Departamento de Matem\'atica,\\
Universidade Federal de S\~ao Carlos,\\
13565-905 S\~ao Carlos, SP, Brazil
}

\maketitle

\begin{abstract}
In this paper we consider PI-algebras $A$ over $\R$ or $\C$. 
It is well known that in general such algebras are not normed algebras. In fact, 
there is a nilpontent commutative algebra which is not a normed algebra, 
see \cite{dales}. 
Here we address the question of whether 
it is possible to find a normed PI-algebra $B$ 
with the same polynomial identities as $A$, and moreover, 
whether there is some Banach PI-algebra with this property.
Our main theorem provides an affirmative answer for this question and moreover
we also show the existence of a Banach Algebra with the same polynomial identities as $A$.  
As a byproduct we prove that if 
$A$ is a normed PI-algebra and its completion is nil, then $A$ is nilpotent. 
By introducing the concept of multihomogeneous norm we obtain as an application
of our main results that if $\FX$ is multihomogeneus normed algebra and $A$ is a PI-algebra such that
the completion of the quotient space $\FX/Id(A)$ is nil, then $A$ is nilpotent. 
Both applications are extensions of the study initiated in \cite{grabiner}.
\end{abstract}

\smallskip


\smallskip


\section{Introduction}

We begin this article by stating precisely some of our main results  
and then we proceed to introduce the concept of multihomogeneous norm. In order to be concise and objective, we will skip the precise definition of some of the basic concepts needed here, such as PI-algebra and Normed algebra. These, together with other additional background concepts, will appear in detail in Section 2. The proofs of the results stated here are found in the last section.

Let $F$ be the field $\R$ or $\C$ and $F\langle X \rangle$ the free 
non-unitary associative algebra, freely generated over $F$ by the infinite set
$X=\{x_1,x_2,\ldots\}$. All the algebras considered in this paper will be 
non-unitary, associative and over the field $F$. 
Thus for convenience we will only use the term \textit{algebra}. 
A good example to keep in mind is the algebra $F\langle X \rangle$. 

For a normed algebra $A$ we write $C(A)$ to denote its completion. 
If $A$ is a PI-algebra then we denote by $Id(A)$ the set of all polynomial identities of $A$.

The statement of our first result is: 

\begin{proposition}\label{identities-of-C(A)}
If $A$ is a normed PI-algebra, then $Id(A)=Id(C(A))$.
\end{proposition}

In other words this proposition tell us that every normed PI-algebra $A$ has the same polynomial identities that some Banach PI-algebra. Since not all PI-algebras are normed PI-algebras, see for example \cite{dales}, a natural question to ask is: given a PI-algebra $A$, is there some
Banach PI-algebra $B$ with the same polynomial identities of $A$ ? 
As we said before we give an affirmative answer for this question and we also show how to construct such algebra $B$.
To explain the construction we introduce some definitions. 

Let $f=f(x_1, \dots, x_n) \in \FX$ be a polynomial, which will be written as
\[
f = \displaystyle\sum_{d_1\ge 0, \dots, d_n \ge 0}f^{(d_1, \dots, d_n)}\ ,
\]
where $f^{(d_1, \dots, d_n)}=f^{(d_1, \dots, d_n)}(x_1, \dots, x_n)$ is the multihomogeneous component of $f$ with multidegree $(d_1,\ldots,d_n)$.

\begin{definition}
A norm $|| \cdot ||$ in $\FX$ is called multihomogenous if 
\[
||f^{(d_1,\ldots,d_n)}|| \leq ||f||
\]
for all $f=f(x_1,\ldots,x_n) \in \FX$ and all $d=(d_1,\ldots,d_n)$. 
If $\FX$ is a normed algebra with respect to a 
multihomegeneous norm, then we say that $\FX$ is a MN-algebra.
\end{definition}

An example of MN-algebras can be obtained as follows. Take $f=\sum_m \alpha_m m$,
where $\alpha_m \in F$ and $m$ is a monomial, then $\FX$ with the norm 
\[||f||=\sum_m |\alpha_m|.\]
is a MN-algebra.

In the sequel, we prove that if $\FX$ is a MN-algebra and if $A$ is a PI-algebra, 
then $Id(A)$ is a closed ideal of $\FX$. Thus the multihomogeneous norm in $\FX$ induces a 
norm in the quotient algebra $\FX / Id(A)$ by
\[||f+Id(A)||=\mbox{inf} \{||f+g|| \, : \, g\in Id(A)\},\]
where $f\in \FX$. We remark that the quotient $\FX / Id(A)$ is a normed algebra 
with this norm and using this fact we obtain our second main result:

\begin{theorem}\label{identities-in-banach}
Let $\FX$ be a MN-algebra. If $A$ is a PI-algebra, then
\[
Id(A)=Id\left(C\left(\frac{\FX}{Id(A)}\right)\right).
\]
\end{theorem}
In particular, a PI-algebra has the same polinomial identities that some Banach PI-algebra.
As an application, we obtain similar results as in the Grabiner's paper. 
In \cite{grabiner} the author proves the following:

\begin{theorem} \label{banachnilnilpotente}
Let $A$ be a Banach algebra. If $A$ is nil then $A$ is nilpotent.
\end{theorem}

In the above theorem, the algebra $A$ is required to be a Banach algebra. 
Here we investigate when the nilpotency of an algebra $A$ 
can be obtained by hypothesis imposed on $C(A)$. In this direction 
our first result is  
\begin{corollary} \label{CAnilthenAnilpot}
Let $A$ be a normed PI-algebra. If $C(A)$ is nil, then $A$ is nilpotent.
\end{corollary}
Our second result relates the nilpotency of $A$ to the completion of certain quotient space related to
the polynomial identities of $A$. To be more 
precise we prove the following:   
\begin{corollary} \label{quotnilthenAnilp}
Let $\FX$ be a MN-algebra and let $A$ be a PI-algebra. If 
\[
C\left(\frac{\FX}{Id(A)}\right)
\]
is nil, then $A$ is nilpotent.
\end{corollary}

\section{Banach and PI-Algebras}

An algebra $A$ is said to be \textit{normed} if it satisfies the followings properties:
\begin{itemize}
\item[a)] $A$ has a norm $\| \cdot \|$;

\item[b)] $\|ab\| \le \|a\|\|b\|$ for all $a, b \in A$.
\end{itemize}

A normed algebra $A$ is called \textit{Banach algebra} if $A$ is a complete normed space.    
It's well known that every normed algebra $A$ is contained in some Banach algebra $C(A)$ such that $A$ is dense in $C(A)$. 
This algebra is the \textit{completion}
of $A$. Here we recall its construction: we first define the relation $\sim$ in the set of all Cauchy sequences of $A$ by
\[(a_n)_n \sim (b_n)_n \Longleftrightarrow \lim_{n \rightarrow \infty} \|a_n-b_n\|=0.\]
Denote by $C(A)$ the set of all equivalence classes. If $(a_n)_n$ is a Cauchy sequence in $A$, we denote by 
$[(a_n)_n]$ its equivalence class. The algebraic operations in $C(A)$ are defined as usual:
$[(a_n)_n]+[(b_n)_n]=[(a_n+b_n)_n]$, 
for any $\lambda \in F$, we put $\lambda [(a_n)_n]=[(\lambda a_n)_n]$ and
$[(a_n)_n][(b_n)_n]=[(a_nb_n)_n]$. Endowed with these operations and with the norm
\[\|[(a_n)_n]\|=\lim_{n \rightarrow \infty} \|a_n\|\]
we have that $C(A)$ is a Banach algebra. Note that an element $a\in A$ is identified with the class 
$[(a)_n]\in C(A)$ of the constant sequence equal to $a$. For more details see \cite{naimark,Kaplansky-Metric}.
 
Let $F\langle X \rangle$ be the free non-unitary associative algebra, freely generated over $F$ by the infinite set
$X=\{x_1,x_2,\ldots\}$. The elements of $\FX$ are called \textit{polynomials} and a polynomial of the kind
$x_{i_1}x_{i_1}\ldots x_{i_n}$ is called \textit{monomial}.

A polynomial $f(x_1, \dots, x_m) \in \FX$ is called a \textit{polynomial identity} for an algebra $A$
if $f(a_1, \dots, a_m) = 0$, for all $a_1,\ldots,a_m \in A$. We denote by $Id(A)$ the set of all polynomial identities of $A$. If 
$Id(A) \neq \{0\}$ then we say that $A$ is a \textit{PI-algebra}.
The set $Id(A)$ is an ideal in $\FX$ and has the property
$f(g_1,\ldots,g_m) \in Id(A)$
for all $f(x_1,\ldots,x_m) \in Id(A)$ and $g_1,\ldots,g_m \in F\langle X \rangle$. Thus we say that
$Id(A)$ is a \textit{T-ideal}.
For details, see \cite{drenskybook,gzbook}.

Let $\FX^{(d_1,\ldots,d_m)}$ be the vector subspace of $\FX$ spanned by all monomials 
$u=x_{j_1}\ldots x_{j_t}$, where the variable $x_i$ appears $d_i$ times in $u$ for all $i=1,\ldots, m$.
If $f(x_1,\ldots, x_m) \in \FX^{(d_1,\ldots,d_m)}$ then we say that $f$ is \textit{multihomegenous} of
\textit{multidegree} $(d_1,\ldots,d_m)$.
Note that if $f=f(x_1, \dots, x_m) \in \FX$, we can always write
\[
f = \displaystyle\sum_{d_1\ge 0, \dots, d_m \ge 0}f^{(d_1, \dots, d_m)}
\]
where $f^{(d_1, \dots, d_m)} \in \FX^{(d_1, \dots, d_m)}$.
The polynomials $f^{(d_1,\ldots,d_m)}$ are called 
the \textit{multihomogenous components} of $f$.

Now we recall an important result that will be used in the next section. 

\begin{theorem}\label{identpolhomopol}
If $f=f(x_1, \dots, x_m)$ is a 
polynomial identity for an algebra $A$, 
then every multihomogeneous component of $f$ is a polynomial identity for $A$.
\end{theorem}
\begin{proof}
See \cite{gzbook}, Theorem 1.3.2. 
\end{proof}

We call the reader's attention to the fact that the above result holds for every infinite field $F$, but it is not always true for finite fields.

\section{Proofs of the Main Results}

In this section we give the proofs of the main results of the paper.

\begin{proof}(Proposition \ref{identities-of-C(A)})
Since $A \subseteq C(A)$, it follows that $Id(C(A)) \subseteq Id(A)$. 
Let $f(x_1,\ldots,x_m) \in Id(A)$ and  $a_1,\ldots,a_m \in C(A)$. 
By the construction of $C(A)$ we have  
 \[f(a_1,\ldots,a_m)=f([(a_{1n})_n],\ldots,[(a_{mn})_n])=[(f(a_{1n},\ldots,a_{mn}))_n]=[(0)_n].\]
The last identity implies that $f\in Id(C(A))$. Therefore $Id(A) \subseteq Id(C(A))$.
\end{proof}

Let $g(x_1,\ldots,x_t)$ a polynomial and let $g^{(d_1,\ldots,d_t)}$ its multihomogeneous component
of multidegree $(d_1,\ldots,d_t)$. If $m < t$ and $d_{m+1}=d_{m+2}=\ldots =d_{t}=0$, then we write
\[g^{(d_1,\ldots,d_m,\ldots,d_t)}=g^{(d_1,\ldots,d_m)}.\]
If $t < m$, then we write
\[g^{(d_1,\ldots,d_t)}=g^{(d_1,\ldots,d_t,0,\ldots,0)},\]
where the number of zeros is $m-t$. With this convention we have the following lemma:

\begin{lemma}\label{multihomogenous-convergence}
Let $\FX$ be a MN-algebra and let $f=f(x_1,\ldots,x_m)$ be a polynomial.
If $(f_n)_n$ is a sequence in $\FX$ such that $f_n \to f$, then
\[f_n^{(d_1,\ldots,d_m)} \to f^{(d_1,\ldots,d_m)}\]
for all multidegree $d=(d_1,\ldots,d_m)$.
\end{lemma}

\begin{proof}
 Once $\| \cdot \|$ is a multihomogeneous norm we have the following inequality
\[
\| f_n^{(d_1, \ldots, d_m)} - f^{(d_1, \ldots, d_m)}\| \le \|f_n - f\|.
\]
Since $f_n \to f$ it follows that $f_n^{(d_1,\ldots,d_m)} \to f^{(d_1,\ldots,d_m)}$.
\end{proof}

\begin{proposition}\label{prop-multigraded}
Let $A$ be a PI-algebra. If $\FX$ is a MN-algebra, then $Id(A)$ is a closed ideal.
\end{proposition}

\begin{proof} 
Let $(f_n)_n \in Id(A)$ be a sequence of polynomials such that $f_n \to f$. We want to prove that $f\in Id(A)$.
Write
\[f=f(x_1,\ldots,x_m)= \sum_{(d_1, \dots, d_m)}f^{(d_1, \dots, d_m)}.\]
By the Lemma \ref{multihomogenous-convergence}, we have that  
$f_n^{(d_1, \dots, d_m)} \to f^{(d_1, \dots, d_m)}$ for all multidegree $(d_1, \dots, d_m)$. Note that
$f_n^{(d_1, \dots, d_m)}\in \FX^{(d_1, \dots, d_m)}\cap Id(A)$ by the Theorem \ref{identpolhomopol}.

Since $\FX^{(d_1, \dots, d_m)}$ is a finite-dimensional vector space
follows that 
\[\FX^{(d_1, \dots, d_m)} \cap Id(A)\] has also finite dimension. 
Since every finite-dimensional space is closed in the norm topology, 
we have that $\FX^{(d_1, \dots, d_m)} \cap Id(A)$ is closed.
Thus  
\[f^{(d_1, \dots, d_m)} \in \FX^{(d_1, \dots, d_m)} \cap Id(A)\]
and therefore $f \in Id(A)$.
\end{proof}

If $A$ is a PI-algebra and $\FX$ is MN-algebra, then by the above proposition we can define a norm in 
the quotient algebra $\FX/Id(A)$:
\[\|f+Id(A)\|=\mbox{inf} \{\|f+g\| \, : \, g\in Id(A)\},\]
where $f\in \FX$. With this norm we have that $\FX/Id(A)$ is a normed algebra.
So we can see that the Theorem \ref{identities-in-banach} 
describes the polynomial identities of the completion of this quotient algebra.
Now we proceed to its proof.

\begin{proof}(Theorem \ref{identities-in-banach}) 
By a classical result in PI-Algebra we have  
\[
Id(A)=Id\left(\frac{\FX}{Id(A)}\right),
\]
see \cite{gzbook}. So the proof of the theorem follows immediately from Proposition \ref{identities-of-C(A)}.
\end{proof}

\begin{proof}(Corollary \ref{CAnilthenAnilpot})
If $C(A)$ is nil, then by Theorem \ref{banachnilnilpotente}, we have that $C(A)$ is nilpotent. 
Thus $x_1x_2\ldots x_n$ is a polynomial identity of $C(A)$ for some $n$.
Since by Proposition \ref{identities-of-C(A)} we have $Id(A) = Id(C(A))$, 
follows that $A$ is nilpotent.
\end{proof}

\begin{proof}(Corollary \ref{quotnilthenAnilp})
Let $B=\FX/Id(A)$. If $C(B)$ is nil, then by Theorem \ref{banachnilnilpotente} we have that $C(B)$ is nilpotent.
Thus $x_1x_2\ldots x_n$ is polynomial identity of $C(B)$ for some $n$.
Since by Theorem \ref{identities-in-banach} we have $Id(A) = Id(C(B))$, follows that $A$ is nilpotent.
\end{proof}


\begin{thebibliography}{99}

\bibitem{dales}
H. G. Dales, \textit{Norming Nil Algebras}, Proc. Amer. Math. Soc.,  \textbf{83}, Number 1, 71--74, 1981.

\bibitem{drenskybook}
V. Drensky, \textit{Free algebras and PI-algebras}, Graduate Course in
Algebra, Springer, Singapore, 1999.

\bibitem{gzbook}
A. Giambruno, M. Zaicev, \textit{Polynomial identities and asymptotic
   methods}, Math. Surveys Monographs \textbf{122}, AMS, Providence, RI,
2005.

\bibitem{grabiner}
S. Grabine, \textit{The nilpotency of Banach nil algebras}, Proc. Amer. Math. Soc., \textbf{21}, 510, 1969.

\bibitem{Kaplansky-Metric} I. Kaplansky, \textit{Set Theory and Metric Spaces},
Allyn and Bacon Series in Advanced Mathematics, Boston, Mass., 1972.

\bibitem{naimark}
M. A. Naimark, \textit{Normed algebra}, Wolters-Noordhoff, 3 edition, 1972.
 
\end{thebibliography}
\end{document}